\documentclass{hbam}
\usepackage{latexsym}
\usepackage{amsthm,amsmath,amssymb,amstext}

\usepackage{graphicx} 
\usepackage{psfrag,color}

\let\Large=\large

\def\XXint#1#2#3{{\setbox0=\hbox{$#1{#2#3}{\int}$}
     \vcenter{\hbox{$#2#3$}}\kern-.5\wd0}}

\newcommand{\R}{\mathbb{R}}
\newcommand{\pr}[2]{\textup{P}_{#1}\left(#2\right)}

\newcommand{\set}[2]{\left\{#1\,\left|\;#2\right.\right\}}

\newcommand{\mint}[3]{\int_{#1}#2\,\textup{d}#3}
\newcommand{\Pab}[1]{\textup{P}_{[a,b]}\left(#1\right)}
\newcommand{\indi}[1]{{\chi}}

\newtheorem{Theorem}{Theorem}[section]

\newtheorem{Lemma}[Theorem]{Lemma}

\newtheorem{LemmaAndDef}[Theorem]{Lemma and Definition}

\newtheorem{Corollary}[Theorem]{Corollary}

\theoremstyle{definition}

\newtheorem{Example}[Theorem]{Example}

\numberwithin{equation}{section}

\def\mathref#1{\ifmmode\mathrm{(\ref{#1})}\else(\ref{#1})\fi}

\def\rhx2{\sqrt{1+ r_{h,x}^2}}

\def\L2{{L^2(\Gamma)}}
\def\Lz{{L^2_0(\Gamma)}}
\def\Lzd{{L^2_{0}(\Gamma^h)}}
\def\Ld{{L^2(\Gamma^h)}}
\def\Hone{{H^1(\Gamma)}}
\def\Htwo{{H^2(\Gamma)}}
\def\C1{{C^1(\Gamma)}}

\def\Uad{{U_{ad}}}
\def\Uadh{{U_{ad} ^h}}
\def\Ltwoh{{L^2(\Gamma^h)}}
\def\C{c_\textup{int}}
\def\CF{C_\textup{FE}}

\def\ldown{{(\cdot)_l}}
\def\lup{{(\cdot)^l}}

\begin{document}

\title{Optimal Control of the Laplace-Beltrami operator on compact surfaces -- concept and numerical treatment}
\author{Michael Hinze and Morten Vierling}

\Nr{2011-1}%

\date{January 2011}

\maketitle
\newpage

\begin{center}
 {\bf \LARGE  }
 \end{center}
 

\pagenumbering{arabic}

{\small {\bf Abstract:} 
We consider optimal control problems of elliptic PDEs on hypersurfaces $\Gamma$ in $\R^n$ for $n=2,3$. The leading part of the PDE is given by the Laplace-Beltrami operator, which is discretized by finite elements on a polyhedral approximation of $\Gamma$. The discrete optimal control problem is formulated on the approximating surface and is solved numerically with a semi-smooth Newton algorithm. We derive optimal a priori error estimates for problems including control constraints and provide numerical examples confirming our analytical findings.
} \\[2mm]
 
\section{Introduction}\label{S:Intro}

We are interested in the numerical treatment of the following linear-quadratic optimal control problem on a $n$-dimensional, sufficiently smooth hypersurface $\Gamma\subset \R^{n+1}$, $n=1,2$.

\begin{equation}\label{E:OptProbl}
\begin{split}
\min_{u\in\L2,\,y\in \Hone} J(u,y)&=\frac{1}{2}\|y-z\|_\L2^2+\frac{\alpha}{2}\|u\|_\L2^2\\
\text{subject to }\quad& u\in U_{ad} \textup{ and }\\
& \mint{\Gamma}{\nabla_\Gamma y \nabla_\Gamma \varphi+\mathbf cy\varphi}{\Gamma} =\mint{\Gamma}{u\varphi}{\Gamma}\,,\forall\varphi\in\Hone
\end{split}
\end{equation}
with   $U_{ad}=\set{v\in \L2}{a\le v\le b}$, $a<b\in\R$ . For simplicity we will assume $\Gamma$ to be compact and $\mathbf c=1$. In  section \ref{Ciszero} we briefly investigate the case $\mathbf c=0$, in section \ref{S:Examples} we give an example on a surface with boundary.

Problem \eqref{E:OptProbl} may serve as a mathematical model for the optimal distribution of surfactants on a biomembrane $\Gamma$ with regard to achieving a prescribed desired concentration $z$ of a quantity $y$.

It follows by standard arguments that \eqref{E:OptProbl} admits a unique solution $u\in\Uad$ with unique associated state $y=y(u)\in\Htwo$.

Our numerical approach uses variational discretization applied to \eqref{E:OptProbl}, see \cite{Hinze2005} and  \cite{HinzePinnauUlbrich2009}, on a discrete surface $\Gamma^h$ approximating $\Gamma$. The discretization of the state equation in \eqref{E:OptProbl} is achieved by the finite element method proposed in \cite{Dziuk1988}, where a priori error estimates for finite element approximations of the Poisson problem for the Laplace-Beltrami operator are provided. Let us mention that uniform estimates are presented in  \cite{Demlow2009}, and steps towards a posteriori error control for elliptic PDEs on surfaces are taken by Demlow and Dziuk in \cite{DemlowDziuk2007}.
For alternative approaches for the discretization of the state equation by finite elements see the work of Burger \cite{Burger2008}. Finite element methods on moving surfaces are developed by Dziuk and Elliott in \cite{DziukElliott2007}.
To the best of the authors knowledge, the present paper contains the first attempt to treat optimal control problems on surfaces.
\vspace{1ex}

We assume that $\Gamma$ is of class $C^2$ with unit normal field $\nu$. As an embedded, compact hypersurface in $\mathbb R^{n+1}$ it is orientable and hence the zero level set of a signed distance function $|d(x)| = \textup{dist}(x,\Gamma)$. We assume w.l.o.g. $\nabla d(x)=\nu(x)$ for $x\in\Gamma$.
Further, there exists an neighborhood $\mathcal N\subset\R^{n+1}$ of $\Gamma$, such that $d$ is also of class $C^2$ on $\mathcal N$ and the projection 

\begin{equation}\label{E:Projection}a:\mathcal N\rightarrow\Gamma\,,\quad a(x) = x-d(x)\nabla d(x)\end{equation} 
is unique, see e.g.  \cite[Lemma 14.16]{GilbargTrudinger1998}. Note that $\nabla d (x)=\nu(a(x))$.

Using $a$ we can extend any function $\phi:\Gamma\rightarrow\R$ to $\mathcal N$ as $\bar \phi(x) = \phi(a(x))$. This allows us to represent the surface gradient in global exterior coordinates $\nabla_\Gamma \phi=(I-\nu\nu^T)\nabla \bar\phi$, with the euclidean projection $(I-\nu\nu^T)$ onto the tangential space of $\Gamma$. 

We use the Laplace-Beltrami operator $\Delta_\Gamma = \nabla_\Gamma\cdot\nabla_\Gamma$ in its weak form i.e. $\Delta_\Gamma:H^1(\Gamma)\rightarrow {H^1}(\Gamma)^{*}$
$$
y\mapsto - \mint{\Gamma}{\nabla_\Gamma y \nabla_\Gamma (\,\cdot\,)}{\Gamma}\in {H^1}(\Gamma)^*\,.
$$

Let $S$ denote the prolongated restricted solution operator of the state equation
\begin{equation*}
S:\L2\rightarrow \L2\,,\quad u\mapsto y\qquad-\Delta_\Gamma y + \mathbf c y=u\,,
\end{equation*}
 which is compact and constitutes a linear homeomorphism onto $\Htwo$, see \cite[1. Theorem]{Dziuk1988}.

By standard arguments we get the following necessary (and here also sufficient) conditions for optimality of $u\in \Uad$
\begin{equation}\label{I:NecCond}
\langle \nabla_u J(u,y(u)), v-u\rangle_\L2=\langle \alpha u+ S^*(S u-z), v-u\rangle_\L2\ge 0\,\quad\forall v\in \Uad\,,
\end{equation}
We rewrite \eqref{I:NecCond} as
\begin{equation}\label{E:NecCond1}
u = \pr{\Uad}{-\frac{1}{\alpha} S^*(S u-z)}\,,
\end{equation}
where $\textup{P}_{U_{ad}}$ denotes the $L^2$-orthogonal projection onto $\Uad$. 

\section{Discretization}
We now discretize \eqref{E:OptProbl} using an approximation $\Gamma^h$ to $\Gamma$ which is globally of class $C^{0,1}$. Following Dziuk, we consider polyhedral $\Gamma^h=\bigcup_{i\in I_h}T_h^i$ consisting of triangles $T_h^i$ with corners  on $\Gamma$, whose maximum diameter is denoted by $h$. With FEM error bounds in mind we assume the family of triangulations $\Gamma^h$ to be regular in the usual sense that the angles of all triangles are bounded away from zero uniformly in $h$.

We assume for $\Gamma^h$ that $a(\Gamma^h)=\Gamma$, with $a$ from \eqref{E:Projection}. For small $h>0$ the projection $a$ also is injective on $\Gamma^h$. In order to compare functions defined on $\Gamma^h$ with functions on $\Gamma$ we use $a$ to lift a function $y\in \Ltwoh$ to $\Gamma$ 
$$
y^l(a(x))=y(x)\quad \forall x\in \Gamma^h\,,
$$
and for $y\in\L2$ and sufficiently small $h>0$ we define the inverse lift
$$
y_l(x)=y(a(x))\quad \forall x\in \Gamma^h\,.
$$
For small mesh parameters $h$ the lift operation $\ldown: \L2\rightarrow \Ltwoh$ defines a linear homeomorphism with inverse $\lup$. Moreover, there exists $\C>0$ such that 
\begin{equation}\label{E:IntBnd}
1-\C h^2\le \|\ldown\|_{\mathcal L(\L2,\Ld)}^2,\|\lup\|_{\mathcal L(\Ld,\L2)}^2\le 1+ \C h^2\,,\end{equation}
 as the following lemma shows.
\begin{LemmaAndDef}\label{L:Integration}
Denote by  $\frac{\textup{d}\Gamma}{\textup{d}\Gamma^h}$ the Jacobian of $a|_{\Gamma^h}:\Gamma^h\rightarrow\Gamma$, i.e. $\frac{\textup{d}\Gamma}{\textup{d}\Gamma^h}=|\textup{det}(M)|$ where $M\in\R^{n\times n}$ represents the Derivative $\textup da(x):T_x\Gamma^h\rightarrow T_{a(x)}\Gamma$ with respect to arbitrary orthonormal bases of the respective tangential space. For small $h>0$ there holds
$$
\sup_\Gamma \left| 1-\frac{\textup{d}\Gamma}{\textup{d}\Gamma^h}\right|\le \C h^2\, ,
$$
Now let $\frac{\textup{d}\Gamma^h}{\textup{d}\Gamma}$ denote $|\textup{det}(M^{-1})|$, so that by the change of variable formula
$$
\left|\mint{\Gamma^h}{v_l}{\Gamma^h}-\mint{\Gamma}{v}{\Gamma}\right|=\left|\mint{\Gamma}{v\frac{\textup{d}\Gamma^h}{\textup{d}\Gamma}-v}{\Gamma}\right|\le \C h^2\|v\|_{L^1(\Gamma)}\,.
$$
\end{LemmaAndDef}
\begin{proof} see \cite[Lemma 5.1]{DziukElliott2007}\end{proof}

Problem \eqref{E:OptProbl} is approximated by the following sequence of optimal control problems
\begin{equation}\label{E:OptProblDiscr}
\begin{split}
\min_{u\in\Ld,\,y\in H^1(\Gamma^h)} J(u,y)&=\frac{1}{2}\|y-z_l\|_\Ld^2+\frac{\alpha}{2}\|u\|_{L^2(\Gamma^h)}^2\\
\text{subject to }\quad&u\in U_{ad} ^h\textup{ and }\\
&y= S_h u\,,
\end{split}
\end{equation}
with $U_{ad}^h=\set{v\in \Ld}{a\le v\le b}$, i.e. the mesh parameter $h$ enters into $\Uad$ only  through $\Gamma^h$ . Problem \eqref{E:OptProblDiscr} may be regarded as the extension of variational discretization introduced in \cite{Hinze2005} to optimal control problems on surfaces.

In \cite{Dziuk1988} it is explained, how to implement a discrete solution operator $S_h:\Ld\rightarrow \Ld$, such that
\begin{equation}\label{E:FEconvergence}
\|\lup S_h\ldown- S\|_{\mathcal L(\L2,\L2)}\le \CF h^2\,,
\end{equation}
which we will use throughout this paper. See in partikular \cite[Equation (6)]{Dziuk1988} and \cite[7. Lemma]{Dziuk1988}. 
For the convenience of the reader we briefly sketch the method. Consider the space $$V_h=\set{\varphi\in C^0\left(\Gamma^h\right)}{\forall i\in I_h:\:\varphi|_{T_h^i}\in\mathcal P^1(T_h^i)}\subset H^1(\Gamma^h)$$ of piecewise linear, globally continuous functions on $\Gamma^h$. For some $u\in\L2$, to compute $y_h^l = \lup S_h\ldown u$  solve
\begin{equation*}
\mint{\Gamma^h}{\nabla_{\Gamma^h}y_h\nabla_{\Gamma^h}\varphi_i+\mathbf cy_h\varphi_i}{\Gamma^h}=\mint{\Gamma^h}{u_l\varphi_i}{\Gamma^h}\,,\quad \forall \varphi\in V_h
\end{equation*}
for $y_h\in V_h$. 
We choose $\Ld$ as control space, because in general we cannot evaluate $\mint{\Gamma}{v}{\Gamma}$ exactly, whereas the expression $\mint{\Gamma^h}{v_l}{\Gamma^h}$ for piecewise polynomials $v_l$ can be computed up to machine accuracy. Also, the operator $S_h$ is self-adjoint, while $(\lup S_h\ldown)^*=\ldown^*S_h\lup^*$ is not. The adjoint operators of $\ldown$ and $\lup$ have the shapes
\begin{equation}\label{E:AdjLift}
\forall v\in\Ld:\:(\ldown)^*v = \frac{\textup{d}\Gamma^h}{\textup{d}\Gamma}v^l\,,\quad\forall v\in\L2:\:(\lup)^*v = \frac{\textup{d}\Gamma}{\textup{d}\Gamma^h}v_l\,,
\end{equation}
hence  evaluating $\ldown^*$ and $\lup^*$ requires knowledge of  the Jacobians $\frac{\textup{d}\Gamma^h}{\textup{d}\Gamma}$ and $\frac{\textup{d}\Gamma}{\textup{d}\Gamma^h}$ which may not be known analytically.

Similar to \eqref{E:OptProbl}, problem \eqref{E:OptProblDiscr} possesses a unique solution $u_h\in\Uadh$ which satisfies
\begin{equation}\label{E:NecCond2_Discr}
u_h = \pr{U_{ad}^h}{-\frac{1}{\alpha}p_h(u_h)}\,.
\end{equation}
Here $ P_{U_{ad}^h}:\Ld\rightarrow U_{ad}^h$ is the $\Ld$-orthogonal projection onto $\Uadh$ and for $v\in \Ld$  the adjoint state is $p_h(v)= S_h^*(S_h v-z_l)\in H^1(\Gamma^h)$.

\vspace{1ex}

Observe that the projections $\textup{P}_\Uad$ and $\textup{P}_\Uadh$ coincide with the point-wise projection $\textup{P}_{[a,b]}$ on $\Gamma$ and $\Gamma^h$, respectively, and hence
\begin{equation}\label{E:ProjRel}
\left(\pr{\Uadh}{v_l}\right)^l=\pr{\Uad}{v}
\end{equation}
 for any $v\in \L2$.
 
 Let us now investigate the relation between the optimal control problems \eqref{E:OptProbl} and \eqref{E:OptProblDiscr}.
\begin{Theorem}[Order of Convergence]\label{T:Convergence}
Let $u\in\L2$, $u_h\in\Ld$ be the solutions of \eqref{E:OptProbl} and \eqref{E:OptProblDiscr}, respectively. Then  for sufficiently small $ h>0$ there holds
\begin{equation}\label{E:ErrorEstimate}\begin{split}
\alpha \big\|u^l_h-u\big\|_\L2^2+\big\|y^l_h-y\big\|_\L2^2\le\frac{1+\C h^2}{1-\C h^2}\bigg(\frac{1}{\alpha}\left\|\left(\lup S_h^*\ldown- S^*\right)(y-z)\right\|_\L2^2&\dots\\
+\left\|\left(\lup S_h\ldown-S\right)u\right\|_\L2^2\bigg)&\,,
\end{split}\end{equation}
with $y=Su$ and $y_h=S_hu_h$.
\end{Theorem}
\begin{proof} 
From \eqref{E:ProjRel} it follows that the projection of $-\left(\frac{1}{\alpha}p(u)\right)_l$ onto $U_{ad}^h$ is $u_l$
$$
u_l = \pr{U_{ad}^h}{-\frac{1}{\alpha}p(u)_l}\,,
$$
which we insert into the necessary condition of \eqref{E:OptProblDiscr}. This gives
$$
\langle \alpha u_h + p_h( u_h),u_l-u_h\rangle_\Ld\ge 0\,.
$$
On the other hand $u_l$ is the $\Ld$-orthogonal projection of $-\frac{1}{\alpha}p(u)_l$, thus
$$
\langle-\frac{1}{\alpha}p(u)_l-u_l,u_h-u_l\rangle_\Ld\le 0\,.
$$
Adding these inequalities yields
\begin{equation*}
\begin{split}
\alpha\|u_l- u_h\|^2_\Ld\le&\langle\left(p_h( u_h)-p(u)_l\right),u_l- u_h\rangle_\Ld\\
=&\langle p_h(u_h) -S_h^*(y-z)_l,u_l-u_h\rangle_\Ld + \langle S_h^*(y-z)_l-p(u)_l,u_l-u_h\rangle_\Ld\,.
\end{split}
\end{equation*}
The first addend is estimated via
\[\begin{split}
\langle p_h(u_h) -S_h^*(y-z)_l,u_l-u_h\rangle_\Ld &= \langle y_h-y_l,S_hu_l-y_h\rangle_\Ld\\
&=-\|y_h-y_l\|^2_\Ld +\langle y_h-y_l,S_hu_l-y_l\rangle_\Ld\\
&\le - \frac{1}{2}\|y_h-y_l\|^2_\Ld+ \frac{1}{2}\|S_hu_l-y_l\|^2_\Ld\,.
\end{split}\]
The second addend satisfies
\[
\langle S_h^*(y-z)_l-p(u)_l,u_l-u_h\rangle_\Ld\le \frac{\alpha}{2}\|u_l-u_h\|^2_\Ld+ \frac{1}{2\alpha}\|S_h^*(y-z)_l-p(u)_l\|_\Ld^2\,.
\]
Together this yields
\[
\alpha\|u_l- u_h\|^2_\Ld + \|y_h-y_l\|^2_\Ld\le \frac{1}{\alpha}\|S_h^*(y-z)_l-p(u)_l\|_\Ld^2 + \|S_hu_l-y_l\|^2_\Ld
\]
The claim follows using \eqref{E:IntBnd} for sufficiently small $h>0$. 
\end{proof}

Because both $S$ and $S_h$ are self-adjoint, quadratic convergence follows directly from \eqref{E:ErrorEstimate}. For operators that are not self-adjoint one can use
\begin{equation}\label{E:FEadjointconvergence}
\|\ldown^* S_h^*\lup^*- S^*\|_{\mathcal L(\L2,\L2)}\le \CF h^2\,.
\end{equation}
which is a consequence of \eqref{E:FEconvergence}.
Equation \eqref{E:AdjLift} and Lemma \ref{L:Integration} imply
\begin{equation}\label{E:LiftAdjConvergence}
\|(\ldown)^*-\lup\|_{\mathcal L(\Ld,\L2)}\le\C h^2\,,\quad \|(\lup)^*-\ldown\|_{\mathcal L(\L2,\Ld)}\le\C h^2\,.
\end{equation}
Combine  \eqref{E:ErrorEstimate} with \eqref{E:FEadjointconvergence} and \eqref{E:LiftAdjConvergence}  to proof quadratic convergence for arbitrary linear elliptic state equations.

\section{Implementation}\label{S:Implementation}
In order to solve \eqref{E:NecCond2_Discr} numerically, we proceed as in \cite{Hinze2005} using the finite element techniques for PDEs on surfaces developed in \cite{Dziuk1988} combined with the semi-smooth Newton techniques from \cite{HintermuellerItoKunisch2003} and \cite{Ulbrich2003} applied to the equation
\begin{equation}\label{E:NecCondAlg}
G_h(u_h)=\left(
u_h - \pr{[a,b]}{-\frac{1}{\alpha} p_h(u_h)}\,.
\right)=0
\end{equation}
Since the operator $p_h$ continuously maps $v\in\Ld$ into $H^1(\Gamma^h)$, Equation \eqref{E:NecCondAlg} is semismooth and thus is amenable to a semismooth Newton method. 
The generalized derivative of $G_h$ is given by
\begin{equation*}
DG_h(u)=\left(I + \frac{\indi{u,m}}{\alpha} S_h^*S_h          
\right)\,,
\end{equation*}
where $\indi{}:\Gamma^h\rightarrow\{0,1\}$ denotes the indicator function of the inactive set $\mathcal I(-\frac{1}{\alpha}p_h(u))=\set{\gamma\in\Gamma^h}{a<-\frac{1}{\alpha}p_h(u)[\gamma]<b}$
$$
 \indi{u,m}=\left\{\begin{array}{l} 1\textup{ on } \mathcal I(-\frac{1}{\alpha}p_h(u))\subset\Gamma^h\\
 0 \textup{ elsewhere on }\Gamma^h\end{array}\right.\,,
$$
which we use both as a function and as the operator  $\indi{u,m}:\Ld\rightarrow\Ld$ defined as the point-wise multiplication with the function $\indi{u,m}$. 
A step semi-smooth Newton method for \eqref{E:NecCondAlg} then reads
\[
\left( I + \frac{\indi{u,m}}{\alpha}S_h^*S_h \right) u^+   =-G_h(u)+DG_h(u) u =\pr{[a,b]}{-\frac{1}{\alpha} p_h(u)}+ \frac{\indi{u,m}}{\alpha} S_h^*S_hu\,.
\]

Given $u$ the next iterate $u^+$ is computed by performing three steps
\begin{enumerate}
\item Set $(\left(1-\indi{u}\right)u^+)[\gamma]=\left((1-\indi{u}) \pr{[a,b]}{-\frac{1}{\alpha}p_h(u)+m}\right)[\gamma]$, which is either $a$ or $b$, depending on $\gamma\in\Gamma_h$.
\item Solve 
\begin{equation*}
 \left(I + \frac{\indi{u}}{\alpha} S_h^*S_h \right)\indi{u}u^+={\frac{\indi{u}}{\alpha}\Big(S_h^*z_l-S_h^*S_h\left(1-\indi{u}\right)u^+\Big)}
\end{equation*}for $\indi{u}u^+$ by CG iteration over $L^2(\mathcal I(-\frac{1}{\alpha}p_h(u))$.
\item Set $u^+=\indi{u}u^++(1-\indi{u})u^+\,.$
\end{enumerate}
Details can be found in \cite{HinzeVierling2010} .

\section{The case $\mathbf c=0$}\label{Ciszero}
In this section we investigate the case $\mathbf c=0$ which corresponds to a stationary, purely diffusion driven process. Since $\Gamma$ has no boundary, in this case total mass must be conserved, i.e. the state equation admits a solution only for controls with mean value zero. For such a control the state is uniquely determined up to a constant.
Thus the admissible set $\Uad$ has to be changed to
$$U_{ad}=\set{v\in \L2}{a\le v\le b}\cap\Lz\,,\textup{ where }\Lz:=\set{v\in \L2}{\mint{\Gamma}{v}{\Gamma}=0}\,,$$
and $a<0<b$. Problem \eqref{E:OptProbl} then admits a unique solution $(u,y)$ and there holds $\mint{\Gamma}{y}{\Gamma}=\mint{\Gamma}{z}{\Gamma}$. W.l.o.g we assume $\mint{\Gamma}{z}{\Gamma}=0$ and therefore  only need to consider states with mean value zero. The state equation now reads $y=\tilde Su$ with the solution operator $\tilde S :\Lz\rightarrow\Lz$ of the equation $-\Delta_\Gamma y=u$, $\mint{\Gamma}{y}{\Gamma}=0$.

Using the injection $\Lz\stackrel{\imath}{\rightarrow}\L2$,  $\tilde S$ is prolongated as an operator $S:\L2\rightarrow\L2$ by $S=\imath \tilde S\imath^*$. The adjoint $\imath^*:\L2\rightarrow \Lz$ of $\imath$ is the $L^2$-orthogonal projection onto $\Lz$.
The unique solution of  \eqref{E:OptProbl}  is again characterized by \eqref{E:NecCond1}, where the orthogonal projection now takes the form
\begin{equation*}
 \pr{\Uad}{v}= \Pab{v+m}\,
 \end{equation*}
 with $m\in\R$ chosen such that 
 \begin{equation*}
 \mint{\Gamma}{ \Pab{v+m}}{\Gamma}=0\,.
 \end{equation*}
If for $v\in\L2$ the inactive set $\mathcal I(v+m)=\set{\gamma\in\Gamma}{a<v[\gamma]+m<b}$ is non-empty, the constant m = m(v) is uniquely determined by $v\in\L2$. Hence, the solution  $u\in\Uad$ satisfies
\begin{equation*}
u = \pr{[a,b]}{-\frac{1}{\alpha}p(u)+m\left(-\frac{1}{\alpha}p(u)\right)}\,,
\end{equation*}
with $p(u)= S^*(S u- \imath^*z)\in\Htwo$ denoting the adjoint state and $m(-\frac{1}{\alpha}p(u))\in\R$ is implicitly  given by $\mint{\Gamma}{u}{\Gamma}=0$. Note that $\imath^*\imath$ is the identity on $\Lz$.

In \eqref{E:OptProblDiscr} we now replace $\Uadh$ by $U_{ad}^h=\set{v\in \Ld}{a\le v\le b}\cap\Lzd$. Similar as in \eqref{E:NecCond2_Discr}, the unique solution $u_h$ then satisfies
\begin{equation}\label{E:NecCond2_DiscrZero}
u_h = \pr{U_{ad}^h}{-\frac{1}{\alpha}p_h(u_h)}=\pr{[a,b]}{-\frac{1}{\alpha}p_h(u_h)+m_h\left(-\frac{1}{\alpha}p_h(u_h)\right)}\,,
\end{equation}
with $p_h(v_h)= S_h^*(S_h v_h-\imath^*_hz_l)\in H^1(\Gamma^h)$ and $m_h(-\frac{1}{\alpha}p_h(u_h))\in\R$ the unique constant such that  $\mint{\Gamma^h}{u_h}{\Gamma^h}=0$. Note that $m_h\left(-\frac{1}{\alpha}p_h(u_h)\right)$ is semi-smooth with respect to $u_h$ and thus Equation \eqref{E:NecCond2_DiscrZero} is amenable to a semi-smooth Newton method.

The discretization error between the problems \eqref{E:OptProblDiscr} and \eqref{E:OptProbl} now decomposes into two components, one introduced by the discretization of $U_{ad}$ through the discretization of the surface, the other by discretization of $S$. 

For the first error we need to investigate the relation between $\pr{\Uadh}{u}$ and $\pr{\Uad}{u}$, which is now slightly more involved than in \eqref{E:ProjRel}.
\begin{Lemma}\label{L:mConvergence}
Let $h>0$ be sufficiently small. There exists a constant $C_m>0$ depending only on $\Gamma$, $|a|$ and $|b|$ such that for all $v\in\L2$ with $\mint{\mathcal I(v+m(v))}{}{\Gamma}>0$ there holds 
$$
|m_h(v_l)-m(v)|\le \frac{C_m}{\mint{\mathcal I(v+m(v))}{}{\Gamma}}h^2\,.
$$
\end{Lemma}
\begin{proof}
For $v\in\L2$, $\epsilon>0$ choose $\delta>0$ and $h>0$ so small that  the set
$$\mathcal I_{v}^\delta=\set{\gamma\in\Gamma^h}{a+\delta\le v_l(\gamma)+m(v)\le b-\delta}\,.$$
satisfies $\mint{\mathcal I_{v}^\delta}{}{\Gamma^h}(1+\epsilon)\ge\mint{\mathcal I(v+m(v))}{}{\Gamma}$. It is easy to show that hence $m_h(v_l)$ is unique. Set 
$C =   \C  \max(|a|,|b|)\mint{\Gamma}{}{\Gamma}$.
Decreasing $h$ further if necessary ensures
$$
\frac{Ch^2}{\mint{\mathcal I_{v}^\delta}{}{\Gamma^h}}\le (1+\epsilon)\frac{C h^2}{\mint{\mathcal I(v+m(v))}{}{\Gamma}} \le \delta\,.
$$
For $x\in\R$ let
$$
M^{h}_{v}(x) = \mint{\Gamma^h}{\Pab{v_l+x}}{\Gamma^h}\,.
$$
Since $\mint{\Gamma}{\Pab{v+m(v)}}{\Gamma}=0$, Lemma \ref{L:Integration} yields 
$$|M_v^h(m(v))|\le \C  \|\Pab{v+m(v)}\|_{L^1(\Gamma)} h^2\le C h^2\,.$$ 
Let us assume w.l.o.g.  $-C h^2\le M_v^h(m(v))\le 0$. Then
$$
M_v^h\left(m(v)+\frac{Ch^2}{\mint{\mathcal I_{v}^\delta}{}{\Gamma^h}}\right)\ge M_v^h\left(m(v)\right)+ Ch^2 \ge 0
$$
implies $0\le m(v)-m_h(v)\le \frac{Ch^2}{\mint{\mathcal I_v^\delta}{}{\Gamma^h}}\le  \frac{(1+\epsilon) C}{\mint{\mathcal I(v+m(v))}{}{\Gamma}}h^2$, since $M^h_v(x)$ is continuous with respect to  $x$. This proves the claim.
\end{proof}
Because $$\left(\pr{\Uadh}{v_l}\right)^l-\pr{\Uad}{v}=\pr{[a,b]}{v+m_h(v_l)}-\pr{[a,b]}{v+m(v)}\,,$$ we get the following corollary.
\begin{Corollary}\label{C:ProjError}
Let $h>0$ be sufficiently small and $C_m$ as in Lemma \ref{L:mConvergence}. For any fixed $v\in\L2$ with $\mint{\mathcal I(v+m(v))}{}{\Gamma}>0$ we have
\begin{equation*}
 \left\| \left(\pr{\Uadh}{v_l}\right)^l-\pr{\Uad}{v}\right\|_\L2\le C_m\frac{\sqrt{\mint{\Gamma}{}{\Gamma}}}{\mint{\mathcal I(v+m(v))}{}{\Gamma}}h^2\,.
\end{equation*}
\end{Corollary}
Note that since for $u\in\L2$ the adjoint  $p(u)$ is a continuous  function on $\Gamma$, the corollary is applicable for $v=-\frac{1}{\alpha}p(u)$.

The following theorem can be proofed along the lines of Theorem \ref{T:Convergence}.
\begin{Theorem}
Let $u\in\L2$, $u_h\in\Ld$ be the solutions of \eqref{E:OptProbl} and \eqref{E:OptProblDiscr}, respectively, in the case $\mathbf c=0$. Let $\tilde u_h=\left(\pr{\Uadh}{-\frac{1}{\alpha}p(u)_l}\right)^l$. Then there holds for $\epsilon>0$ and $0\le h< h_\epsilon$
\[
\begin{split}
\alpha\|u^l_h-\tilde u_h\|_\L2^2+\big\|y^l_h-y\big\|_\L2^2\le(1+\epsilon)\bigg(\frac{1}{\alpha}\left\|\left(\lup S_h^*\ldown- S^*\right)(y-z)\right\|_\L2^2&\dots\\
+\left\|\lup S_h\ldown\tilde u_h-y\right\|_\L2^2\bigg)&\,.
\end{split}
\]
\end{Theorem}
Using Corollary \ref{C:ProjError} we conclude from the theorem
\[\begin{split}
\|u^l_h-u\|_\L2\le& C\left(\frac{1}{\alpha}\left\|\bigg(\lup S_h^*\ldown- S^*\right)(y-z)\right\|_\L2 + \frac{1}{\sqrt{\alpha}}\left\|\left(\lup S_h\ldown-S\right)u\right\|_\L2\dots\\
&+\left(1+\frac{\|S\|_{\mathcal L(\L2,\L2)}}{\sqrt{\alpha}}\right) \frac{C_m\sqrt{\mint{\Gamma}{}{\Gamma}}\,h^2}{\mint{\mathcal I(-\frac{1}{\alpha}p(u)+m(-\frac{1}{\alpha}p(u)))}{}{\Gamma}}\bigg)\,,
\end{split}\]
the latter part of which is the error introduced by the discretization of $\Uad$.
Hence one has  $h^2$-convergence of the optimal controls.
\section{Numerical Examples}\label{S:Examples}
The figures show some selected  Newton steps $u^+$. Note that jumps of the color-coded function values are well observable along the border between active and inactive set. For all examples Newton's method is initialized with $u_0\equiv 0$.

The meshes are generated from a macro triangulation through congruent refinement, new nodes are projected onto the surface $\Gamma$. The maximal edge length $h$ in the triangulation is not exactly halved in each refinement, but up to an error of order $O(h^2)$. Therefore we just compute our estimated order of convergence (EOC) according to
\[
EOC_i = \frac{\ln \|u_{h_{i-1}}-u_l\|_{L^2(\Gamma^{h_{i-1}})}-\ln \|u_{h_i}-u_l\|_{L^2(\Gamma^{h_{i}})}}{\ln(2)}.
\]
For different refinement levels, the tables show $L^2$-errors, the corresponding EOC and the number of Newton iterations before the desired accuracy of $10^{-6}$ is reached. 


It was shown in \cite{HintermuellerUlbrich2004}, under certain assumptions on the behaviour of  $-\frac{1}{\alpha} p(u)$, that the undamped Newton Iteration is mesh-intdependent.  These assumptions are met by all our examples, since the surface gradient of  $-\frac{1}{\alpha} p(u)$ is bounded away from zero along the border of the inactive set. Moreover, the displayed number of Newton-Iterations suggests mesh-independence of the semi-smooth Newton method. 

\begin{Example}[Sphere I]\label{Ex:Sphere_two}
We consider the problem 
\begin{equation}\label{P:SphereRegular}
\min_{u\in\L2,\,y\in \Hone} J(u,y) \text{ subject to }
-\Delta_\Gamma y + y=u - r,\quad  -1\le u\le1
\end{equation}
with $\Gamma$ the unit sphere in $\R^3$ and $\alpha = 1.5\cdot10^{-6}$. We choose $z=52\alpha x_3(x_1^2-x_2^2)$ , to obtain the solution
\[
\bar u = r = \min\big(1,\max\big(-1,4x_3(x_1^2-x_2^2)\big)\big)
\]
of \eqref{P:SphereRegular}.
\end{Example}

\begin{figure}
\center
\includegraphics[width =\textwidth]{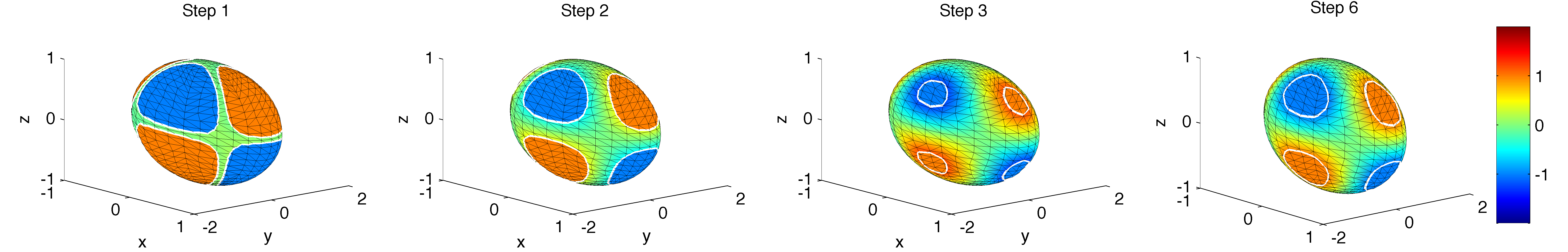}
\caption{Selected full Steps $u^+$ computed for Example \ref{Ex:Sphere_two} on the twice refined sphere.}
\end{figure}

\begin{table}
\begin{tabular}{r|cccccc}
reg. refs. & 0 & 1 & 2 & 3 & 4 & 5 \\
\hline
 $L2$-error  & 5.8925e-01 &  1.4299e-01  & 3.5120e-02  & 8.7123e-03  & 2.2057e-03 & 5.4855e-04\\ 
 EOC & -    &  2.0430  &  2.0255   & 2.0112   & 1.9818   & 2.0075\\
 \# Steps & 6   &  6   &  6 &    6 &    6 &    6\\
 \end{tabular}
 \caption{$L^2$-error, EOC and number of iterations for Example \ref{Ex:Sphere_two}.}
 \end{table}

\begin{Example}\label{Ex:Graph}
Let $\Gamma = \set{(x_1,x_2,x_3)^T\in\R^3}{x_3=x_1x_2\land x_1,x_2\in(0,1)}$ and $\alpha =10^{-3}$. For
\begin{equation*}
\min_{u\in\L2,\,y\in \Hone} J(u,y) \text{ subject to }
-\Delta_\Gamma y =u - r,\quad y=0\text{ on }\partial\Gamma\quad  -0.5\le u\le0.5
\end{equation*}
we get
\[
\bar u = r = \max\big(-0.5,\min\big(0.5,\sin(\pi x)\sin(\pi y)\big)\big)
\]
by proper choice of $z$ (via symbolic differentiation).
\end{Example}
 
 \begin{figure}
\center
\includegraphics[width =\textwidth]{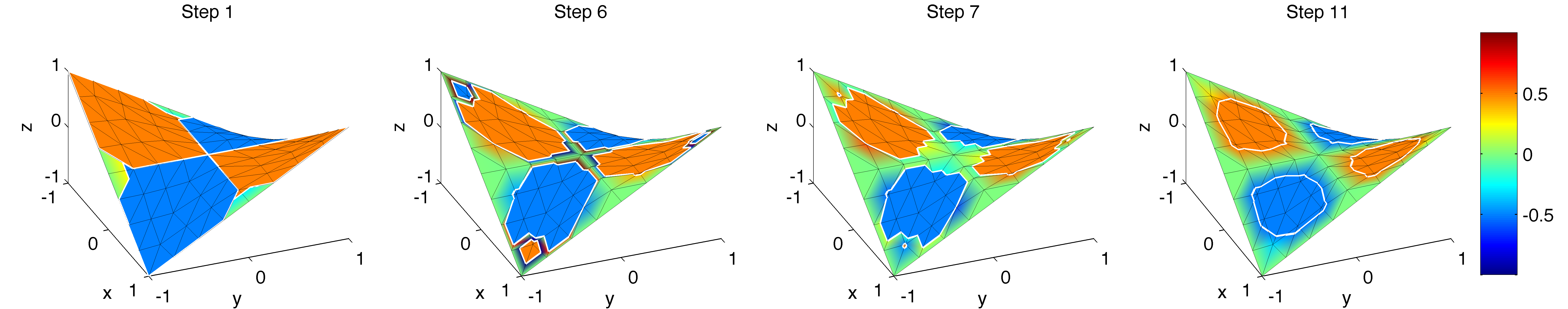}
\caption{Selected full Steps $u^+$ computed for Example \ref{Ex:Graph} on the twice refined grid.}
\end{figure}

\begin{table}
\begin{tabular}{r|cccccc}
reg. refs. & 0 & 1 & 2 & 3 & 4 & 5 \\
\hline
 $L2$-error  & 3.5319e-01 &  6.6120e-02  & 1.5904e-02  & 3.6357e-03 &8.8597e-04  & 2.1769e-04\\ 
 EOC & -    & 2.4173   & 2.0557 &   2.1291 &   2.0369  &  2.0250\\
 \# Steps & 11  &  12 &   12  &  11 &   13  &  12\\
 \end{tabular}
 \caption{$L^2$-error, EOC and number of iterations for Example \ref{Ex:Graph}.}
 \end{table}

Example \ref{Ex:Graph}, although $\mathbf c=0$, is also covered by the theory in Sections \ref{S:Intro}-\ref{S:Implementation}, as by the Dirichlet boundary conditions the state equation remains uniquely solvable for $u\in\L2$.
In the last two examples we apply the variational discretization to optimization problems, that involve zero-mean-value constraints as in Section \ref{Ciszero}. 

 \begin{figure}
\center
\includegraphics[width =\textwidth]{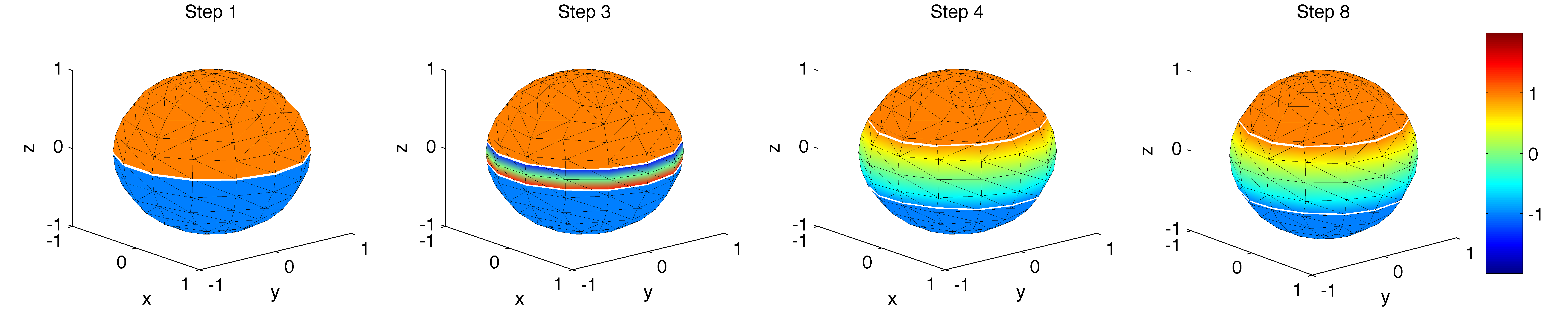}
\caption{Selected full Steps $u^+$ computed for Example \ref{Ex:Sphere} on once refined sphere.}
\end{figure}

\begin{table}
\begin{tabular}{r|cccccc}
reg. refs. &0 & 1 & 2 & 3 & 4 & 5\\
\hline
 $L2$-error  &6.7223e-01 &  1.6646e-01  & 4.3348e-02 &  1.1083e-02 &  2.7879e-03  & 6.9832e-04\\ 
 EOC & -    & 2.0138 &  1.9412  &   1.9677  &   1.9911  &   1.9972 \\
 \# Steps & 8   &  8  &   7 &    7  &   6  &   6\\
 \end{tabular}
 \caption{$L^2$-error, EOC and number of iterations for Example \ref{Ex:Sphere}.}
 \end{table}

 \begin{figure}
\center
\includegraphics[width =\textwidth]{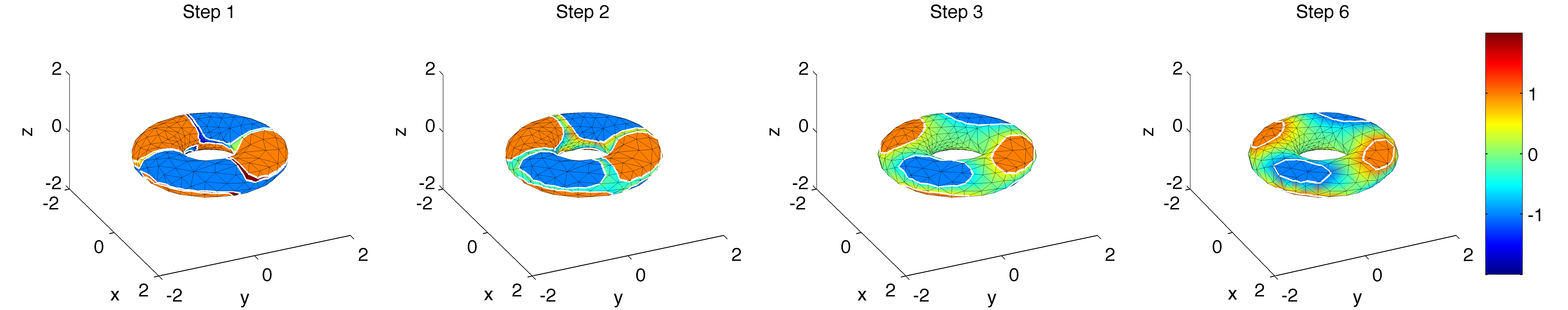}
\caption{Selected full Steps $u^+$ computed for Example \ref{Ex:Torus} on the once refined torus.}
\end{figure}

\begin{table}
\begin{tabular}{r|cccccc}
reg. refs. &0 & 1 & 2 & 3 & 4 & 5\\
\hline
 $L2$-error  &   3.4603e-01 &  9.8016e-02 &  2.6178e-02   &6.6283e-03  & 1.6680e-03  & 4.1889e-04\\ 
 EOC & -    & 1.8198e+00 &  1.9047e+00  & 1.9816e+00 &  1.9905e+00  & 1.9935e+00\\
 \# Steps & 9   &  3    & 3 &    3 &    2   &  2 \\
 \end{tabular}
 \caption{$L^2$-error, EOC and number of iterations for Example \ref{Ex:Torus}.}
 \end{table}

\begin{Example}[Sphere II]\label{Ex:Sphere}
We consider 
\begin{equation*}
\begin{split}
\min_{u\in\L2,\,y\in \Hone} J(u,y)
\text{ subject to } -\Delta_\Gamma y=u\,,\quad -1\le u\le 1\,,\quad \mint{\Gamma}{y}{\Gamma} = \mint{\Gamma}{u}{\Gamma}=0\,,
\end{split}
\end{equation*}
 with $\Gamma$ the unit sphere in $\R^3$. Set $\alpha=10^{-3}$ and 
\[
z(x_1,x_2,x_3) = 4\alpha x_3 + \left\{\begin{array}{rlc} \ln(x_3+1)+C \,,& \textup{if }  &0.5\le x_3\\ 
x_3-\frac{1}{4}\textup{arctanh}(x_3)\,,&\textup{if } &-0.5\le x_3\le 0.5\\
-C-\ln(1-x_3)\,, & \textup{if }& x_3\le -0.5
\end{array}\right.\,,
\]
where $C$ is chosen for $z$ to be continuous.
The solution according to these parameters is
\[
\bar u = \min\big(1,\max\big(-1, 2 x_3\big)\big)\,.
\]
\end{Example}

\begin{Example}[Torus]\label{Ex:Torus}
Let $\alpha=10^{-3}$ and  $$\Gamma = \set{(x_1,x_2,x_3)^T\in\R^3}{\sqrt{x_3^2+\left(\sqrt{x_1^2+x_2^2}-1\right)^2}=\frac{1}{2}}$$ the 2-Torus embedded in $\R^3$. By symbolic differentiation we compute $z$, such that
\[
\min_{u\in\L2,\,y\in \Hone} J(u,y)
\text{ subject to }
-\Delta_\Gamma y=u-r,\quad -1\le u\le 1\,,\quad  \mint{\Gamma}{y}{\Gamma}=\mint{\Gamma}{u}{\Gamma}=0
\]
is solved by
\[
\bar u = r = \max\big(-1,\min\big(1,5xyz\big)\big)\,.
\]
\end{Example}
As the presented tables clearly demonstrate, the examples show the expected convergence behaviour.
\section*{Acknowledgement}
The authors would like to thank Prof. Dziuk for the fruitful discussion during his stay in Hamburg in November 2010.

\bibliographystyle{annotate}
\bibliography{/Users/morten/Documents/Bibs/all}

\begin{thebibliography}{HPUU09}

\bibitem[Bur08]{Burger2008}
M. Burger.
\newblock {Finite element approximation of elliptic partial differential
  equations on implicit surfaces.}
\newblock {\em Comput. Vis. Sci.}, 12(3):87--100, 2008.


\bibitem[DD07]{DemlowDziuk2007}
A. Demlow and G. Dziuk.
\newblock {An adaptive finite element method for the Laplace-Beltrami operator
  on implicitly defined surfaces.}
\newblock {\em SIAM J. Numer. Anal.}, 45(1):421--442, 2007.


\bibitem[DE07]{DziukElliott2007}
G. Dziuk and C.M. Elliott.
\newblock {Finite elements on evolving surfaces.}
\newblock {\em IMA J. Numer. Anal.}, 27(2):262--292, 2007.


\bibitem[Dem09]{Demlow2009}
A. Demlow.
\newblock {Higher-order finite element methods and pointwise error estimates
  for elliptic problems on surfaces.}
\newblock {\em SIAM J. Numer. Anal.}, 47(2):805--827, 2009.


\bibitem[Dzi88]{Dziuk1988}
G. Dziuk.
\newblock {Finite elements for the Beltrami operator on arbitrary surfaces.}
\newblock {Partial differential equations and calculus of variations, Lect.
  Notes Math. 1357, 142-155}, 1988.

\bibitem[GT98]{GilbargTrudinger1998}
D. Gilbarg and N.~S. Trudinger.
\newblock {\em {Elliptic partial differential equations of second order.}}
\newblock {Berlin: Springer}, 1998.


\bibitem[HIK03]{HintermuellerItoKunisch2003}
M.~Hinterm{\"u}ller, K.~Ito, and K.~Kunisch.
\newblock {The primal-dual active set strategy as a semismooth Newton method.}
\newblock {\em SIAM J. Optim.}, 13(3):865--888, 2003.


\bibitem[Hin05]{Hinze2005}
M.~Hinze.
\newblock {A variational discretization concept in control constrained
  optimization: The linear-quadratic case.}
\newblock {\em Comput. Optim. Appl.}, 30(1):45--61, 2005.


\bibitem[HPUU09]{HinzePinnauUlbrich2009}
M.~Hinze, R.~Pinnau, M.~Ulbrich, and S.~Ulbrich.
\newblock {\em {Optimization with PDE constraints.}}
\newblock {Mathematical Modelling: Theory and Applications 23. Dordrecht:
  Springer}, 2009.


\bibitem[HU04]{HintermuellerUlbrich2004}
M. Hinterm{\"u}ller and M. Ulbrich.
\newblock {A mesh-independence result for semismooth Newton methods.}
\newblock {\em Mathematical Programming}, 101:151--184, 2004.


\bibitem[HV11]{HinzeVierling2010}
M. Hinze and M. Vierling.
\newblock A globalized semi-smooth newton method for variational discretization
  of control constrained elliptic optimal control problems.
\newblock To appear in {\em Constrained Optimization and Optimal Control for Partial
  Differential Equations}. Birkh{\"a}user, 2011.

\bibitem[Ulb03]{Ulbrich2003}
M. Ulbrich.
\newblock {Semismooth Newton methods for operator equations in function
  spaces.}
\newblock {\em SIAM J. Optim.}, 13(3):805--841, 2003.


\end{thebibliography}
\end{document}